\documentclass{amsart}
\usepackage{amssymb}
\usepackage{mathrsfs}

\begin{document}

\title{Note as to size-minimal hypercompletely separating systems}

\author{Barbora Bat\'\i kov\'a, Tom\'a\v s J.\  Kepka and Petr C.\ N\v emec}

\address{Barbora Bat\'\i kov\'a, Department of Mathematics, CULS,
Kam\'yck\'a 129, 165 21 Praha 6 - Suchdol, Czech Republic}
\email{batikova@tf.czu.cz}%
\address{Tom\'a\v s J.\ Kepka,  Faculty of Education, Charles University, M.\ Rettigov\'e 4, 116 39 Praha 1, Czech Republic}
\email{kepka@karlin.mff.cuni.cz}

\address{Petr C.\ N\v emec, Department of Mathematics, CULS,
Kam\'yck\'a 129, 165 21 Praha 6 - Suchdol, Czech Republic}
\email{nemec@tf.czu.cz}

\subjclass{05D05, 05B99, 06E99}

\keywords{separating system, size-minimal}

\begin{abstract}If $S$ is a non-empty finite set, $|S|=s$, then a system $\mathscr{A}$ of subsets of $S$ is a size-minimal hypercompletely separable system (i.e., for every $a\in S$ there are $A,B\in\mathscr{A}$ such that $A\cap B=\{a\}$) if and only if $|\mathscr{A}|=\left\lceil\frac{1+\sqrt{8s+1}}2\right\rceil$.
\end{abstract}

\maketitle

\newtheorem{ex}{Example}[section]
\newtheorem{obs}[ex]{Observation}
\newtheorem{remark}[ex]{Remark}
\newtheorem{calc}[ex]{Calculation}
\newtheorem{lemma}[ex]{Lemma}
\newtheorem{theorem}[ex]{Theorem}
\newtheorem{prop}[ex]{Proposition}
\newtheorem{const}[ex]{Construction}
\newtheorem{scholium}[ex]{Scholium}

\section{Introduction}
Separating systems of subsets of finite sets were introduced in \cite{R} and it was readily observed there that a separating system $\mathscr{A}$ (defined on an $s$-element set, $s\ge2$) is size-minimal if and only if $|\mathscr{A}|=\lceil\log_2s\rceil$. A bit later, completely separating systems appeared in \cite{D} and it was found in \cite{S} that such a system $\mathscr{B}$ is size-minimal if and only if $|\mathscr{B}|$ is the smallest positive integer $r$ satisfying the inequality $\binom{r}{\lfloor\frac r2\rfloor}\ge s$. (This achievement is based on a well known Sperner Theorem (\cite{Sp}).) A~particularly remarkable class of hypercompletely separating systems was thoroughly studied in \cite{BFPR} and size-minimal systems of this type were described in \cite{BKN}.

The aim of the present short (general) note is to show that a hypercompletely separating system $\mathscr{C}$ is size-minimal if and only if $|\mathscr{C}|=\left\lceil\frac{1+\sqrt{8s+1}}2\right\rceil$.

\section{Preliminaries}

Throughout this brief note, $S$ is a non-empty finite se,t $s=|S|\ge1$, $\mathcal{P}(S)$ denotes the Booelan algebra of subsets of $S$ and, for each $k$, $0\le k\le s$, $\mathcal{P}_k(S)=\{\,A\in\mathcal{P}(S)\,|\,|A|=k\,\}$. Any subset of $\mathcal{P}(S)$ is called a {\it system} (of subsets of $S$) ($0\le|\mathscr{A}|\le2^s$ for any system $\mathscr{A}$). If $\mathscr{A}$ is a system then $\overline{\mathscr{A}}=\{\,S\setminus A\,|\,A\in\mathscr{A}\,\}$ (we have $|\overline{\mathscr{A}}|=|\mathscr{A}|$). Subsets of $\mathcal{P}(\mathcal{P}(S))$ are called {\it ensembles} (of systems). 

Let $\mathfrak{A}$ be an ensemble of systems. A system $\mathscr{A}\in\mathfrak{A}$ is said to be {\it size-minimal} ({\it size-maximal}, resp.) in $\mathfrak{A}$ if $|\mathscr{A}|\le|\mathscr{B}|$ ($|\mathscr{B}|\le|\mathscr{A}|$, resp.) for any system $\mathscr{B}\in\mathfrak{A}$ and it is said to be {\it inclusion-minimal}  ({\it inclusion-maximal}, resp.) in $\mathfrak{A}$ if $\mathcal{A}=\mathcal{C}$ whenever $\mathcal{A}\in\mathfrak{A}$ is such that $\mathcal{C}\subseteq\mathcal{A}$ ($\mathcal{A}\subseteq\mathcal{C}$, resp.). 

A system $\mathscr{A}$ is called\begin{enumerate}
\item[-] {\it separating} if for each $T\in\mathcal{P}_2(S)$ there is $A\in\mathscr{A}$ such that $|A\cap T|=1$;
\item[-] {\it completely separating} if for each $T\in\mathcal{P}_2(S)$ there are $A,B\in\mathscr{A}$ such that $|A\cap T|=1=|B\cap T|$ and $A\cap T\ne B\cap T$;
\item[-] {\it hypercompletely separating} if for each $a\in S$ there are $A,B\in\mathscr{A}$ such that $A\cap B=\{a\}$ (notice that $A=B$ iff $A=B=\{a\}$).
\end{enumerate}

Obviously, every hypercompletely separating system is a non-empty completely separating system and every completely separating system is separating. Besides, a system $\mathscr{A}$ is (completely, resp.) separating if and only if it generates the Boolean algebra $\mathcal{P}(S)(\cap,\cup,\overline{\phantom{A}})$ (the lattice $\mathcal{P}(S)(\cap,\cup)$, resp.).

In the whole text, an {\it HCSP-system} will mean a hypercompletely separating system (for short).

Obviously, every size minimal HCSP-system is inclusion-minimal, but the converse is not true in general. 

E.g., the system $\{\{1,2,3\},\{1,3,4\},\{1,3,5\},\{1,4,5\},\{2,4,5\},\{3,4,5\}\}$ is an in\-clusion-minimal HCSP-system on $\{1,2,3,4,5\}$ and the system is not size-minimal. Generally, the system $\mathcal{P}_1(S)$ (of one-element subets of $S$) is always an inclusion-minimal HCSP-system, and this system is size-minimal only for $1\le s\le4$.

Let us collect a handful of easy observations  concerning HCSP-systems.

\begin{obs}\label{1.3} \rm {\rm(i)} Let $\mathscr{A}$ be a HCSP-system on $S$. Then $\bigcup\mathscr{A}=S$. If $s=|S|\ge2$ then $\bigcap\mathscr{A}=\emptyset$.\newline
{\rm(ii)} Let $\mathscr{A}\subseteq\mathscr{B}$, $\mathscr{A},\mathscr{B}$ systems on $S$. If $\mathscr{A}$ is HCSP then $\mathscr{B}$ is such.\newline
{\rm(iii)} If $T$ is a non-empty subset of $S$ and $\mathscr{A}$ is HCSP on $S$ then the restriction $\mathscr{A}\uparrow T=\{\,A\cap T\,|\,A\in\mathscr{A}\,\}$ is HCSP on $T$.\newline
{\rm(iv)} Let $S_1,\dots,S_k$, $k\ge1$, be pair-wise disjoint non-empty finite sets, $|S_i|=s_i\ge1$,  and let $\mathscr{A}_i$ be an HCSP-system on $S_i$, $1\le i\le k$. Then $\mathscr{A}=\bigcup\mathscr{A}_i$ is an HCSP-system on $S=\bigcup S_i$ and $|\mathscr{A}|=\sum|\mathscr{A}_i|$. Moreover, the system $\mathscr{A}$ is incusion-minimal on $S$ if and only if all $\mathscr{A}_i$ are inclusion-minimal on $S_i$. Further, $\mathscr{B}=\{\,A_1\times\dots\times A_k\,|\,A_i\in\mathscr{A}_i\,\}$ is an HCSP-system on $T=S_1\times\dots\times S_k$ and $|\mathscr{B}|=\prod|\mathscr{A}_i|$.\newline
{\rm(v)} The system $\mathcal{P}_1(S)$ of one-element subsets is always an inclusion-minimal HCSP-system. On the other hand, the system $\overline{\mathcal{P}_1(S)}$ (=$\mathcal{P}_{s-1}(S)$) is an HCSP-sustem just for $1\le s\le3$.\newline
{\rm(vi)} The system $\mathcal{P}_k(S)$, $0\le k\le s$, is an HCSP system if and only if $1\le k\le\frac{s+1}2$. In particular, $\mathcal{P}_2(S)$ is an HCSPsystem just for $s\ge3$.\end{obs}


\section{Preparatory results}

For every non-negative integer $k$ we denote by $\alpha(k)$ the $k$-th triangular number. I.e., $\alpha(k)=\frac{k(k+1)}2$ $(=\binom{k+1}{2}=\sum_{i=0}^k i$). Clearly, $\alpha(k+1)=\alpha(k)+k+1$.

For every non-negative integer $k$ let $\tau(k)$ designate the smallest (positive) integer such that $\tau(k)\ge\frac{1+\sqrt{8k+1}}2$ ($\tau(k)=\left\lceil\frac{1+\sqrt{8k+1}}2\right\rceil$).

\begin{lemma}\label{2.2} Let $k$ be a non-negative integer. Then:\newline
{\rm(i)} $\tau(\alpha(k))=k+1$ .\newline
{\rm(ii)} If $t$ is an integer such that $\alpha(k)+1\le t\le\alpha(k+1)$ then $\tau(t)=k+2$.\newline
{\rm(iii)} $\tau(k)=\frac{1+\sqrt{8k+1}}2$ if and only if $k=\alpha(l)$ for some $l\ge0$.\newline
{\rm(iv)} $k\ge\tau(k)$ for $k\ge3$, and $k>\tau(k)$ for $k\ge5$.\end{lemma}

\begin{proof} Straightforward.\end{proof}

\begin{prop}\label{2.3} Let $S$ be a finite set, $s=|S|\ge3$. Then $|\mathscr{A}|\ge\tau(s)$ for every HCSP-system on $S$. Moreover, if $|\mathscr{A}|=\tau(s)$ then $\mathscr{A}$ is size-minimal and if $s\ge5$ then $\mathcal{P}_1(S)\cap\mathscr{A}=\emptyset$.\end{prop}

\begin{proof} Put $m=|\mathscr{A}|$, $R=\{\,a\in S\,|\,\{a\}\in\mathscr{A}\,\}$, $r=|R|$, $T=S\setminus R$, $t=|T|$; we have $r+t=s$. The rest of the proof is divided into three parts:\newline
(i) $t=0$. Then $T=\emptyset$, $R=S$, $r=s$, $\mathcal{P}_1(S)\subseteq\mathscr{A}$, $m\ge s$. Since $s\ge3$, we have $m\ge s\ge\tau(s)$ (\ref{2.2}(iv)) ($m>\tau(S)$ for $s\ge5$).\newline
(ii) $r=0$. Then $R=\emptyset$, $T=S$, $t=s$ and $|A|\ge2$ for each $A\in\mathscr{A}$. If $a\in S$ then $\{a\}=A_1\cap A_2$ for some $A_1,A_2\in\mathscr{A}$. Clearly, $A_1\ne A_2$ and we see that ($|\mathcal{P}_2(\mathscr{A})|=\binom{m}{2}=$) $\frac{(m-1)m}2\ge s$, $m^2-m\ge2s$, $(2m-1)^2=4m^2-4m+1\ge8s+1$, $m\ge\left\lceil\frac{1+\sqrt{8s+1}}2\right\rceil=\tau(s)$.\newline
(iii) $r\ne0\ne t$. We get $\mathscr{A}=\mathscr{R}\cap\mathscr{B}$, $\mathscr{R}=\{\,\{a\}\,|\,a\in R\,\}$, $\mathscr{B}=\mathscr{A}\setminus\mathscr{R}$, $m=|\mathscr{A}|=r+q$, $q=|\mathscr{B}|$, $t\le\frac{q(q-1)}2$, $2t\le q^2-q$, $q\ge2$ and we verify that $m>\tau(s)$. Indeed, $r+q=m>\frac{1+\sqrt{8(r+1)+1}}2$ if and only if $(2r+2q-1)^2>8r+8t+1$, which is equivalent to $4r^2+r(8q-12)+4q^2-4q>8t$. As $2t\le q^2-q$, it would be sufficient to check that $4r^2+r(8q-12)+4q^2-4q>4q^2-4q$ or $4r^2+r(8q-12)>0$. But $r\ge1$, $q\ge2$ and $4r^2+r(8q-12)\ge4+4=8$.\end{proof}

\begin{scholium} \rm Let $k\ge0$, $\beta(k)=\frac{1+\sqrt{8k+1}}2$ ($\tau(k)=\lceil\beta(k)\rceil$) and $\gamma(k)=\sqrt{2k}$. We have $\beta(k)^2=2k+\beta(k)=\gamma(k)^2+\beta(k)$. Thus $\beta(k)=\beta(k)^2-\gamma(k)^2=(\beta(k)-\gamma(k))\cdot(\beta(k)+\gamma(k))$. If $k\ge1$ then $1<\gamma(k)<\beta(k)$. Assume, for a~moment, that $\beta(k)\ge\gamma(k)+1$. Then $2k+\beta(k)=\beta(k)^2\ge\gamma(k)^2+2\gamma(k)+1=2k+2\gamma(k)+1$, $\frac{1+\sqrt{8k+1}}2=\beta(k)\ge2\gamma(k)+1=1+\sqrt{8k}$, $\sqrt{8k+1}+1\ge\sqrt{32k}+2$, $\sqrt{8k+1}\ge\sqrt{32k}+1$, $8k+1>32k$, $1>24k$, which is absurd. It means that $1<\gamma(k)<\beta(k)<\gamma(k)+1$.

We have shown that for every $k\ge0$ either $\tau(k)=\lceil\sqrt{2k}\rceil$ or $\tau(k)=\lceil\sqrt{2k}\rceil+1$. The sequence of numbers $k$ such that $\tau(k)=\lceil\sqrt{2k}\rceil$ starts as follows:\newline
$1,3, 5, 6 ,9,10,13,14,15,19,20,21,25,26,27,28,\dots$\,.
\end{scholium}

\begin{prop}\label{3.10} Let $i=1,2,3,4,5,6$ and $S_i=\{1,2,\dots,i\}$, $|S_i|=i$. Then:\newline
{\rm(i)} $\mathcal{P}_1(S_1)=\{\{1\}\}$ is the only size-minimal HCSP-system on $S_1$ and $|\mathcal{P}_1(S_1)|=1<2=\tau(1)$.\newline
{\rm(ii)} $\mathcal{P}_1(S_2)=\{\{1\},\{2\}\}$ is the only size-minimal HCSP-system on $S_2$ and $|\mathcal{P}_1(S_2)|=2<3=\tau(2)$.\newline
{\rm(iii)} $\mathcal{P}_1(S_3)=\{\{1\},\{2\},\{3\}\}$ and $\mathcal{P}_2(S_3)=\overline{\mathcal{P}_1(S_3)}=\{\{1,2\},\{1,3\},\{2,3\}\}$ are the only size-minimal HCSP-systems on $S_3$ and $|\mathcal{P}_1(S_3)|=|\mathcal{P}_2(S_3)|=3=\tau(3)$.\newline
{\rm(iv)} The systems $\mathscr{A}_1=\{\{1\},\{2\},\{3\},\{4\}\}$ ($=\mathcal{P}_1(S_4)$), $\mathscr{A}_2=\{\{1\},\{2,3\},\{2,4\},$ $\{3,4\}\}$, $\mathscr{A}_3=\{\{1,2\},\{1,3\},\{2,4\},\{3,4\}\}$, $\mathscr{A}_4=\{\{1\},\{2,3\},\{3,4\},\{1,2,4\}\}$ and $\mathscr{A}_5=\{\{1,2\},\{1,3\},\{1,4\},\{2,3,4\}\}$ are, up to isomorphism, all size-minimal\linebreak HCSP-systems on $S_4$. In all the cases, we have $|\mathscr{A}|=4=\tau(4)$. Besides, $\overline{\mathscr{A}_2}\cong\mathscr{A}_5$, $\overline{\mathscr{A}_3}\cong\mathscr{A}_3$, $\overline{\mathscr{A}_4}\cong\mathscr{A}_4$ and $\overline{\mathscr{A}_5}\cong\mathscr{A}_2$. On the other hand $\overline{\mathscr{A}_1}=\mathcal{P}_3(S_4)$ is not an HCSP-system.\newline
{\rm(v)} Up to isomorphism, $\mathscr{A}_6=\{\{1,2,3\},\{1,4,5\},\{2,4\},\{3,5\}\}$ is the only size-minimal HCSP-system on $S_5$, $|\mathscr{A}_6|=4=\tau(5)$ and $\overline{\mathscr{A}_6}\cong\mathscr{A}_6$.\newline
{\rm(vi)} Up to isomorphism, $\mathscr{A}_7=\{\{1,2,3\},\{1,4,5\},\{2,4,6\},\{3,5,6\}\}$ is the only size-minimal HCSP-system on $S_6$, $|\mathscr{A}_7|=4=\tau(6)$ and $\overline{\mathscr{A}_7}\cong\mathscr{A}_7$.
\end{prop}

\begin{proof} (i) -- (iv). This is obvious.\newline
(vi) By \ref{2.3}, $\mathscr{A}_7$ is a size-minimal HCSP-system on $S_6$. If $\mathscr{A}$ is a size-minimal HCSP-system on $S_6$ then $|\mathscr{A}|=4$ and $|A|=3$ for every $A\in\mathscr{A}$ by \ref{3.1}(iv). Further, \ref{3.1}(ii) implies that for every $a\in S_6$ there are uniquely determined $A,B\in\mathscr{A}$ such that $A\cap B=\{a\}$, and hence $\mathscr{A}\cong\mathscr{A}_7$.\newline
(v) As $|\mathscr{A}_6|=4=\tau(5)$, $\mathscr{A}_7\uparrow S_5=\mathscr{A}_6$ is a size-minimal HCSP-system on $S_5$ by \ref{2.3}. Let $\mathscr{A}=\{A_1,A_2,A_3,A_4\}$ be a size-minimal HCSP system on $S_5$. Then, e.g., $A_3\cap A_4=\emptyset$ and $\mathscr{B}=\{A_1,A_2,A_3\cup\{6\},A_4\cup\{6\}\}$ is a size-minimal HCSP-system on $S_6$, hence $\mathscr{B}\cong\mathscr{A}_7$ by (vi) and $\mathscr{A}=\mathscr{B}\uparrow S_5\cong\mathscr{A}_6$.\end{proof}

In the rest of this section, let $k\ge3$ and let $S$ be a finite set such that $s=|S|=\alpha(k)$ ($\ge6$). Furthermore, let $\mathscr{A}$ be an HCSP-system (defined on $S$) such that $|\mathscr{A}|=\tau(s)=k+1$ (cf. \ref{2.2}(i)).

\begin{prop}\label{3.1} {\rm(i)} $\mathscr{A}$ is size-minimal (on $S$).\newline
{\rm(ii)} For every $a\in S$ there is a uniquely determined pair $\{A,B\}$ $(\in\mathcal{P}_2(\mathscr{A}))$ such that $A\cap B=\{a\}$.\newline
{\rm(iii)} $|A\cap B|=1$ for all $A,B\in\mathscr{A}$, $A\ne B$.\newline
{\rm(iv)} $|A|=k$ for every $A\in\mathscr{A}$.\end{prop}

\begin{proof} (i) It follows from \ref{2.3}.\newline
(ii) We have $|\mathcal{P}_2(\mathscr{A}|=\binom{\tau(s)}{2}=\binom{k+1}{2}=\alpha(k)=s=|S|$ (see also \ref{2.3}).\newline
(iii) It follows from (ii).\newline
(iv) If $A\in\mathscr{A}$ then $|\mathscr{A}\setminus\{A\}|=|\mathscr{A}|-1=k$ and, by (iii), $|A\cap B|=1$ for each $B\in\mathscr{A}\setminus\{A\}$. Besides, by (ii), if $B_1,B_2\in\mathscr{B}$, $B_1\ne B_2$, then $A\cap B_1\ne A\cap B_2$. Now, it is visible that $|A|\le k$. Put $A_1=\{\,a_B\,|\,\{a_B\}=A\cap B, B\in\mathscr{A}\setminus\{A\}\,\}$. We have $A_1\subseteq A$ and $|A_1|=k$. If $a\in A$ then $\{a\}=C\cap D$ for suitable $C,D\in\mathscr{A}$. Using (ii), we conclude that $A\in\{C,D\}$, and hence $a\in A_1$. Thus $A_1=A$.\end{proof}

\begin{const}\label{3.8} \rm Let $S=\{a_1,\dots,a_s\}$, $\mathscr{A}=\{A_1,\dots,A_{k+1}\}$, $s=\alpha(k)=\frac{k(k+1)}2$ ($\ge6$). We have $\alpha(k+1)-\alpha(k)=k+1$ and we put $s^+=\alpha(k+1)=s+k+1$, $S^+=S\cup\{a_{s+1},\dots,a_{S^+}\}=\{a_1,\dots,a_{s^+}\}$, $|S^+|=s^+$, $A_1^+=A_1\cup\{a_{s+1}\}$, $A_2^+=A_2\cup\{a_{s+2}\}$, $\dots$, $A_l^+=A_l\cup\{a_{s+l}\}$, $\dots$, $A_{k+1}^+=A_{k+1}\cup\{a_{s+k+1}\}$ ($=A_{k+1}\cup\{a_{s^+}\}$), $A_{k+2}^+=\{a_{s+1},a_{s+2},\dots,a_{s+k+1}\}$,  $s+k+1=s^+$, $1\le l\le k+1$. Furthermore, $\mathscr{A}^+=\{A_1^+,\dots,A_{k+2}^+$, $|\mathscr{A}^+|=k+2=\tau(\alpha(k+1))=\tau(s^+)$.\newline
(i) $|A_i^+|=k+1$, $1\le i\le k+2$.\newline
(ii) $\mathscr{A}^+$ is an HCSP-system on $S^+$.

Indeed, let $a\in S^+$. If $a=a_i$, $1\le i\le s$, then $\{a_1\}=A_u\cap A_v$, $1\le u<v\le k+1$, and we see immediately that $\{a\}=A_u^+\cap A_v^+$. If $a=a_i$, $s+1\le i\le s+k+1$ ($=s^+$) then $\{a\}=A_{i-s}^+\cap A_{k+2}^+$.\newline
(iii) We have $|\mathscr{A}^+|=\tau(s^+)$, $|S^+|=s^+$ and it follows from \ref{2.3} that $\mathscr{A}^+$ is a~size-minimal HCSP-system on $S^+$.\end{const}

\begin{lemma}\label{3.6} Let $k\ge4$, $r$ be an integer such that $\alpha(k-1)+1=\frac{k^2-k+2}2\le r\le\frac{k^2+k}2=\alpha(k)=s$ and let $R$ be any subset of $S$ such that $|R|=r$ ($\ge7$). Put $\mathscr{B}=\mathscr{A}\uparrow R$ ($=\{\,A\cap R\,|\,A\in\mathscr{A}\,\}$). Then $|\mathscr{B}|=k+1=\tau(r)=|\mathscr{A}|$ and $\mathscr{B}$ is a~size-minimal HCSP-system on $R$.\end{lemma}

\begin{proof} If $a\in R$ then $\{a\}=A\cap B=(A\cap R)\cap(B\cap R)$ for some $A,B\in\mathscr{A}$. It follows that the system $\mathscr{B}$ is an HCSP-system on $R$. By \ref{2.3} and \ref{2.2}(ii), $|\mathscr{B}|\ge\tau(r)=k+1=\tau(s)$. On the other hand, $|\mathscr{B}|\le|\mathscr{A}|=\tau(s)=\tau(r)$. Thus $|\mathscr{B}|=\tau(r)$ and $\mathscr{B}$ is size-minimal on $R$ by \ref{2.3}.\end{proof}

\section{Main results}

\begin{theorem}\label{5.1} Let $S$ be a finite set such that $s=|S|\ge3$. An hypercompletely separating system $\mathscr{A}$ defined on $S$ is size-minimal if and only if $|\mathscr{A}|=\left\lceil\frac{1+\sqrt{8s+1}}2\right\rceil$.\end{theorem}

\begin{proof} In view of \ref{2.3}, it suffices to find at last one HCSP-system $\mathscr{A}$ on $S$ such that $|\mathscr{A}|=\tau(s)$. If $3\le s\le6$ then the particular examples are shown in \ref{3.10}.

Further, let $s=\alpha(k)$, $k\ge3$. Taking into account \ref{3.10} and \ref{3.8}, we proceed by induction.

Finally, if ($\alpha(k)+1=$) $\frac{k^2+k+2}2\le s\le\frac{k^2+3k}2$ ($=\alpha(k+1)-1$), $k\ge3$, then the desired example is found in \ref{3.6}.\end{proof}

\begin{remark} \rm The foregoing result is not true for $s=1,2$ (see \ref{3.10}(i),(ii)).\end{remark}

\begin{theorem}\label{5.4} Let $k\ge3$ and $s=|S|=\alpha(k)$ $(=\frac{k(k+1)}2\ge6)$. Then all size-minimal HCSP systems (defined on $S$) are isomorphic.\end{theorem}

\begin{proof}Let $\mathscr{A}$ be a size-minimal HCSP-system (on $S$). By \ref{5.1}, $|\mathscr{A}|=\tau(s)=k+1$. By \ref{3.1}(iii), $|A\cap B|=1$ for all $A,B\in\mathscr{A}$, $A\ne B$. Now, due to \ref{3.1}(ii), the dual system $\{A\in\mathscr{A}\,|\,a\in A\,\}$ is nothing else than the system $\mathcal{P}_2(\mathscr{A})$ (of two-element subsets of $\mathscr{A}$); we have $|\mathcal{P}(\mathscr{A})|=\binom{\tau(s)}{2}=\binom{\tau(\alpha(k)}{2}=\binom{k+1}{2}=\alpha(k)=s$. Now, if $\mathscr{B}$ is another size-minimal HCSP-system (on $S$) then the respective dual systems are isomorphic, and hence the systems are isomorphic as well.\end{proof}

\begin{remark} \rm The foregoing result remains true for $S=\alpha(1)=1$, but fails for $s=\alpha(2)=3$ (see \ref{3.10}(i),(iii)).\end{remark}

\end{document}